\documentclass[12pt,a4paper,leqno]{amsart}

\usepackage[top=1.1in, bottom=1.3in, left=1.5in, right=1.5in]{geometry}

\usepackage[ansinew]{inputenc}
\usepackage[T1]{fontenc}
\usepackage[english]{babel}
\usepackage{amsthm}
\usepackage{amsfonts}         
\usepackage{amsmath}
\usepackage{amssymb}
\usepackage{breqn}
\usepackage{bm}
\usepackage{url}
\usepackage{esint}

\newcommand{\R}{\mathbb{R}}
\newcommand{\C}{\mathbb{C}}

\newcommand{\N}{\mathbb{N}}

\newcommand{\E}{\mathbb{E}}

\theoremstyle{plain}
\newtheorem{theorem}{Theorem}
\newtheorem{lemma}[theorem]{Lemma}
\newtheorem{prop}[theorem]{Proposition}
\newtheorem{cor}[theorem]{Corollary}

\theoremstyle{remark}
\newtheorem{remark}{Remark}

\numberwithin{equation}{section}

\begin{document}

\title{Vector-valued local approximation spaces}

\author{Tuomas Hyt\"onen and Jori Merikoski}

\address{Department of Mathematics and Statistics, P.O.B. 68 (Gustaf H\"allstr\"omin katu 2b), FI-00014 University of Helsinki, Finland}
\email{tuomas.hytonen@helsinki.fi}
\email{jori.merikoski@gmail.com}

\thanks{Both authors were partially supported by the ERC Starting Grant ``AnProb'' (grant no. 278558) and the Finnish Centre of Excellence in Analysis and Dynamics Research (grant no. 271983).}



\begin{abstract}
We prove that for every Banach space $Y$, the Besov spaces of functions from the $n$-dimensional Euclidean space to $Y$ agree with suitable local approximation spaces with equivalent norms. In addition, we prove that the Sobolev spaces of type $q$ are continuously embedded in the Besov spaces of the same type if and only if $Y$ has martingale cotype $q$. We interpret this as an extension of earlier results of Xu (1998), and Mart\'inez, Torrea and Xu (2006). These two results combined give the characterization that $Y$ admits an equivalent norm with modulus of convexity of power type $q$ if and only if weakly differentiable functions have good local approximations with polynomials.
\end{abstract}

\maketitle

\tableofcontents

\section{Introduction}
Let $(Y, \|\cdot\|_Y)$ be a Banach space, and denote the unit ball in $Y$ by $B_Y := \{ y \in Y: \, \|y\|_Y \leq 1 \}.$ Recall that the norm $\|\cdot\|_Y$ (or the space $Y$) is uniformly convex if for every $\epsilon \in (0,2]$ there exists $\delta \in (0,1]$ such that 
\begin{align*}
\text{if } x,y \in B_Y, \,\, \|x-y\|_Y \geq \epsilon, \, \text{ then } \|x+y\|_Y \leq 2(1-\delta).
\end{align*}
The supremum over all such $\delta$ is denoted by $\delta_Y(\epsilon),$ and is called the \emph{modulus of convexity.} We say that the norm $\|\cdot\|_Y$ has \emph{modulus of convexity of power type $q \in [2,\infty),$} if $\delta_Y(\epsilon) \geq \epsilon^q/C^q$ for some uniform constant $C.$ A deep theorem of Pisier (cf. \cite{Pi}) says that a uniformly convex Banach space can always be renormed with an equivalent norm, whose modulus of convexity is of power type $q$ for some $q \in [2, \infty).$

In the pioneering work of Bates, Johnson, Lindenstrauss, Preiss and Schechtman \cite{BJLPS}, the geometry of $Y$ was related to approximation of Lipschitz functions $f:X \to Y$ by affine maps on macroscopically large balls, where $X$ is a finite dimensional Banach space. They defined the \emph{uniform approximation by affine property} of $\text{Lip}(X,Y),$ the space of Lipschitz functions $f:X \to Y,$ and showed that it  is equivalent to the statement that $(Y, \|\cdot\|_Y)$ has an equivalent uniformly convex norm.  One of our main results, Theorem \ref{convex2} below, is of similar flavour; it states that $Y$ admits an equivalent uniformly convex norm with modulus of convexity of power type $q$ if and only if all weakly differentiable functions $f:\R^n \to Y$ have good local approximations with polynomials, or more precisely, the $L^q$-type Sobolev spaces are continuously embedded in suitable local approximation spaces. The case of first order Sobolev spaces and their first order polynomial (i.e. affine) approximation already appeared in the recent works \cite{HLN,HN} as a step in a quantitative elaboration of the mentioned result of Bates et al. While we deal with Sobolev spaces and polynomial approximation of any order, some of the results, particularly on the neccessity of the geometric assumption, are new even in the first order case.

We measure the polynomial approximability  of functions by using a local approximation norm $\| \cdot \|_{A^{sN}_{pq}}$ (cf. notations below), where $s$ is a smoothness parameter, and $N$ is the maximum degree of polynomials allowed. Using this norm, we define the local approximation space by $A^{sN}_{pq}(\R^n;Y).$ It is classical that for $Y=\R, \C,$ the local approximation spaces are related to the \emph{Besov spaces} $B^s_{pq}(\R^n;Y),$ whose norms measure smoothness of a function using a Littlewood-Paley type decomposition. In fact, the classical result states that the two spaces are equal with equivalent norms. See \cite{Tri} (p. 51) for the classical results and their history. There is also a rich literature on embeddings of vector-valued Besov spaces (see e.g. \cite{Amann,KNVW,KS,MV,SSS,SS,Veraar}), but their relation to the local approximation spaces seems not to have been previously addressed, aside from the recent articles \cite{HLN,HN}. We prove in Theorem \ref{AequalsB} that the coincidence of Besov and local approximation spaces extends to the vector-valued setting without any requirement for the Banach space $Y.$ For us it is convenient to define the norm for the Besov space using the heat semigroup $T_t$ (cf. notations below). For $Y=\R, \C,$ using the heat semigroup gives a norm which is equivalent to the usual Besov norm (cf. \cite{Tri}, pp. 52-54, 151-155). This choice of norm allows us to use the semigroup property, and inequalites developed in the recent paper by Hyt\"onen and Naor \cite{HN} to greatly simplify the proofs. As in the first order case studied in \cite{HLN,HN}, this choice also allows to track a good dependence of the various constants on the dimension $n$ of the domain $\R^n,$ an aspect that is new even in the scalar valued case $Y=\R, \C.$

Once we have proved that $A^{sN}_{pq} =B^s_{pq}$ with equivalent norms, proving the promised characterization of uniform convexity of power type $q$ by the continuity of the embedding $W^{k,q} \hookrightarrow A^{kN}_{qq}$ is reduced to studying the embedding $W^{k,q}\hookrightarrow B^{kM}_{qq}.$ This is accomplished in Theorem \ref{convex} below. To prove Theorem \ref{convex}, we need the work of Pisier which relates the geometric properties of $Y$ to inequalities of martingales with values in $Y$ (cf. \cite{Pi}); recall that a sequence $\{M_k\}_{k=1}^{\infty}$ of $Y$-valued functions on a probability space $(\Omega,\mathcal{F}, \mu)$ is called a martingale if there is an increasing sequence of $\sigma$-algebras $\mathcal{F}_1 \subseteq \mathcal{F}_2 \subseteq \cdots \subseteq \mathcal{F}$  such that the conditional expectation $\E[M_{k+1} | \mathcal{F}_k] = M_k$ for all $k \geq 1.$ Then we have the following characterization:
\begin{theorem} \label{convexmart} \emph{(Pisier).} Let $q \in [2, \infty).$ A Banach $Y$ admits an equivalent norm with modulus of convexity of power type $q$ if and only if every martingale $\{M_k\}_{k=1}^{\infty},$ $M_k \in L^q(\Omega,\mu;Y),$ satisfies
\begin{align} \label{mcotype}
\left( \sum_{k=1}^{\infty} \| M_{k+1} - M_k \|^{q}_{L^q(\Omega,\mu;Y)}  \right)^{1/q} \lesssim \, \sup_{k\geq 1} \| M_k \|_{L^q(\Omega,\mu;Y)}. 
\end{align}
\end{theorem}
Following Pisier,  we say that the Banach space has \emph{martingale cotype $q$} if the above martingale inequality holds. Hence, the above theorem states that $Y$ admits an equivalent norm with modulus of convexity of power type $q$ if and only if $Y$ has martingale cotype $q.$ For a Banach space $Y$ of martingale cotype $q,$ we denote by $\mathfrak{m}(q,Y)$ the best constant in the above martingale inequality.

Using Theorem \ref{convexmart}, our task is further reduced to showing that the property of having martingale cotype $q$ can be characterized by the continuity of the embedding $$W^{k,q}(\R^n;Y) \hookrightarrow B^k_{qq}(\R^n;Y).$$ This is closely related to the discovery by Xu (cf. \cite{Xu}) that studying the martingale cotype of $Y$ is equivalent to studying the Littlewood-Paley theory of $Y$-valued functions. In fact, in Section \ref{sobolevsection} we show that the results of Xu in \cite{Xu}, and Mart\'inez, Torrea and Xu in \cite{MTX} (cf. Theorem 5.2 in particular), can be interpreted as the statement that $Y$ has martingale cotype $q$ if and only if $L^q(\R^n;Y) \hookrightarrow B^0_{qq}(\R^n;Y)$ continuously. This corresponds to the case $k=0$ of Theorem \ref{convex}. Thus, the present article extends the work in  the above mentioned papers. The proof of Theorem \ref{convex} relies heavily on the methods developed in \cite{Xu} and \cite{MTX}, as well as the Littlewood-Paley theory developed in \cite{HN}. Here also the definition of the Besov norm using the heat semigroup $T_t$ is very convenient.

\subsection{Notations} \label{notations}

In order to give a precise formulation of our contributions, we must first fix some notations. Let $Y$ be a Banach space, and let $n \in \N,$ $p \in [1,\infty].$   We will be considering mainly functions $f:\R^n \to Y,$ so without risk of confusion we may set $L^p:=L^p(\R^n; Y),$  and $\|\cdot \|_{L^p} := \|\cdot \|_{L^p(\R^n;Y)}$ for the space of $L^p$-functions $f:\R^n \to Y.$ 

Let $p\in[1, \infty ],$ and $q \in [1, \infty).$ For any  $s \in\R,$ and any integer $M > s/2,$ we define the  (homogeneous) Besov norm by
\begin{align*}
[ f ]_{B^{sM}_{pq}} & := \left (\int_0^{\infty}  t^{(M-s/2)q}\left \|   T_t^{(M)} f  \right \|_{L^p}^q \frac{dt}{t} \right)^{1/q},
\end{align*}
where $T_t^{(M)}$ is defined shortly.
For $s> 0$, we define the Besov space
\begin{equation*}
 B^{sM}_{pq}  :=  B^{sM}_{pq}(\R^n; Y) :=\{f \in L^p: \| f \|_{B^{sM}_{pq}}:=\|f\|_{L^p}+[f]_{B^{sM}_{pq}} < \infty\}.
\end{equation*}
(A modification is necessary for $s\leq 0$. Although we still consider the homogeneous norm $[\cdot]_{B^{0M}_{pq}}$ in this case, we refrain from defining the actual space, which in general contains singular distributions, as our results only deal with the intersection $L^p\cap B^{0M}_{pq}$. This is still given by the previous displayed formula with $s=0$.)

Above, $T_t = e^{\Delta t}$ is the heat semigroup, and  $T_t^{(M)} f = \partial_t^M T_t f = \Delta^{M} T_t f.$  We have the integral representation
\begin{align*}
T_t f(x) = \int_{\R^n} k_t(x-y) f(y) dy, \quad k_t(x) := \frac{1}{(4 \pi  t )^{n/2}}e^{-\vert x \vert^2/4t},
\end{align*}
where $k_t$ is the heat kernel.  We refer to \cite{Ste} for the classical theory of the heat semigroup, and to \cite{MTX} for the vector-valued theory. 

Similarly, for any nonnegative integer $N$  and $s>0,$ we define the (homogeneous)  local approximation norm by
\begin{align*}
[ f ]_{A^{sN}_{pq}} &:= \left(  \int_0^{\infty} t^{-sq}  \left \| \left( \inf_{\deg P \leq N} \fint_{B(x,t)} \| f(y)-P(y) \| _Y^p dy \right)^{1/p} \right \|_{L^p}^q \frac{dt}{t} \right) ^{1/q} , 
\end{align*}
where the infimum is over polynomials $P$ such that $\deg P \leq N.$ We define the local approximation space
\begin{equation*}
   A^{sN}_{pq}  := A^{sN}_{pq}(\R^n; Y):=\{f \in L^p: \| f \|_{A^{sN}_{pq}} :=\|f\|_{L^p}+[f]_{A^{sN}_{pq}}< \infty\}.
\end{equation*}
We will prove in Section \ref{localsection} below that for all $M>s/2$ and $N>s-1,$ the norms $\| \cdot \|_{B^{sM}_{pq}}$ and $\| \cdot \|_{A^{sN}_{pq}}$ are equivalent. Thus, we may refer to the underlying spaces by $B^{s}_{pq}$ and $A^{s}_{pq}.$ 

\begin{remark} We could have yet another parameter $u$ in the definition of the local approximation spaces by replacing the $L^p$-average over the balls $B(x,t)$ by an $L^u$-average (cf. \cite{Tri}, sec. 1.7.3, eq. (4)). We will not consider this additional generalization.
\end{remark}

For integers $k\geq 1,$ we define the nonhomogeneous Sobolev norm, and space by
\begin{align*}
\| f \|_{W^{k,q}} &:=\| f \|_{W^{k,q}(\R^n;Y)} :=  
 \sum_{0 \leq \vert \alpha \vert \leq k} \| \partial_x^{\alpha} f\|_{L^q}, \\
W^{k,q}=W^{k,q}(\R^n;Y) & := \{ f \in L^q:  \,  \partial_x^{\alpha} f \, \, \text{exists for all }\vert \alpha \vert \leq k \text{ and }  \| f \|_{W^{k,q}}  \leq \infty\}.
\end{align*}
Here $\alpha = (\alpha_1, \dots, \alpha_n),$ $\vert \alpha \vert =\alpha_1 + \cdots + \alpha_n,$ and the derivatives  $\partial_x^{\alpha} f = \partial_{x_1}^{\alpha_1}  \cdots \partial_{x_n}^{\alpha_n} f  $ are the weak derivatives. For $k=0$ we set $ \| \cdot \|_{W^{0,q}} =\| \cdot \|_{L^q},$ and $W^{0,q} (\R^n; Y ) =   L^q(\R^n; Y ).$

For a Banach space $Y$ of martingale cotype $q,$ we denote by $\mathfrak{m}(q,Y)$ the best constant in the martingale inequality (\ref{mcotype}) above. For asymptotic notations we denote for two positive real valued functions $f$ and $g$ that $f \lesssim g$ if there exist a constant $C$ such that $f \leq Cg.$ We write $f \asymp g$ if $f \lesssim g \lesssim f.$ If the implied constant is allowed to depend on  parameters, we write the parameters in the subscript (e.g. $\lesssim_{s,N,M}$).

\subsection{Results}

In Section \ref{localsection} below, we will prove the following theorem, which states that the Besov spaces agree with the corresponding local approximation spaces with equivalent norms. As mentioned above, this generalizes the classical results for $Y =\R,\C.$
\begin{theorem} \label{AequalsB} Let $Y$ be a Banach space and let $s> 0,$ $p \in [1, \infty],$ $q \in [1, \infty),$ $M > s/2,$ and $N > s-1.$ Then $B^{s}_{pq}(\R^n;Y) = A^{s}_{pq}(\R^n;Y),$ and we have
\begin{align*}
[ f ]_{A^{sN}_{pq}}  \lesssim_{s,N,M} [ f ]_{B^{sM}_{pq}} \lesssim_{s,N,M}  n^{c(N,M,s)}[ f ]_{A^{sN}_{pq}}
\end{align*} 
for $c(N,M,s) = (N+M+s+1)/2.$ 
\end{theorem}
\begin{remark}
For geometric applications it is interesting to know the dependence on the dimension $n$ for the implied constants. For example in \cite{HN} and \cite{HLN}, which consider affine approximations, dependence on $n$ plays a crucial role in the results. For this reason, we have taken care in the proof of Theorem \ref{AequalsB} to obtain a polynomial dependence on $n.$
\end{remark}

Our investigations leave open the corresponding result for the closely related Triebel-Lizorkin spaces $F^{s}_{pq}(\R^n;Y)$ and norms, defined by setting, for $M>s/2>0$,
\begin{align*}
\| f \|_{F^{sM}_{pq}} & :=\|f\|_{L^p}+[f]_{F^{sM}_{pq}},\quad [f]_{F^{sM}_{pq}}:= \left \| \left (\int_0^{\infty}  t^{(M-s/2)q} \|  T_t^{(M)} f  \|_Y^q \frac{dt}{t} \right)^{1/q} \right\|_{L^p}.
\end{align*}
If we replace the local approximation norm respectively by
\begin{align*}
\left \| \left(  \int_0^{\infty} t^{-sq}   \left( \inf_{\deg P \leq N} \fint_{B(x,t)} \| f(y)-P(y) \| _Y^p dy \right)^{1/p} \frac{dt}{t} \right) ^{1/q} \right\|_{L^p},
\end{align*}
then the qualitative result corresponding to Theorem \ref{AequalsB} holds for $F^{s}_{pq}(\R^n;Y),$ $Y = \R, \C$ (cf. \cite{Tri}, pp. 151-155, 186-192). It might be of interest to obtain a vector-valued extension of this result as well, but this seems to be harder to achieve for technical reasons. Another possible extension consists of exponents in the quasi-Banach regime $p\in(0,1)$ and/or $q\in(0,1)$, but we also leave this for future investigations.

In Section \ref{sobolevsection} we change our attention to Sobolev spaces. Using Theorem \ref{convexmart}, we obtain from Theorem \ref{mart} below the following characterization of Banach spaces with an equivalent norm whose modulus of convexity has power type $q.$

\begin{theorem} \label{convex}
Let $Y$ be a Banach space, and $q \in [2, \infty).$ Let $k \geq 0,$ $n\geq 1,$ $M>k/2$ be integers. Then $Y$ admits an equivalent norm with modulus of convexity of power type $q$ if and only if $[f]_{B^{kM}_{qq}(\R^n;Y)}\lesssim\|f\|_{W^{k,q}(\R^n;Y)}$ for all $f\in W^{k,q}(\R^n;Y)$. 
\end{theorem}

Combining Theorem \ref{convex} with Theorem \ref{AequalsB} yields the following characterization, which morally states that a Banach space admits an equivalent norm whose modulus of convexity is  of power type $q$ if and only if weakly differentiable functions admit good local approximations with polynomials.

\begin{theorem} \label{convex2}
Let $Y$ be a Banach space, and let $q \in [2, \infty).$ Let $k \geq 1,$ $n\geq 1,$ $N \geq k$ be integers. Then $Y$ admits an equivalent norm with modulus of convexity of power type $q$ if and only if 
 $[ f ]_{A^{kN}_{qq}(\R^n;Y)}  \lesssim \| f \|_{W^{k,q}(\R^n;Y)}$ for all $f \in W^{k,q}(\R^n;Y).$
\end{theorem}

\begin{remark} For $k =1, N=1,$ the above result is closely related to the open \emph{Question 10} proposed on page 10 of \cite{HN}. The question asks if a similar estimate for compactly supported Lipschitz is sufficient to show that $Y$ has martingale cotype $q.$ However, the estimate for the Lipschitz functions in \cite{HN} is seemingly a weaker condition than $[ f ]_{A^{11}_{qq}}  \lesssim \| f \|_{W^{1,q}},$ since the former follows from the latter.
\end{remark}

\begin{remark} In contrast to Theorem \ref{convex}, we only consider $k\geq 1$ in Theorem \ref{convex2}, since we have only defined and studied the local  approximation spaces with strictly positive smoothness $s >0.$ The same restriction also appears in the classical theory (cf. \cite{Tri}).
\end{remark}

See Theorem \ref{mart}, and Corollary \ref{mart2} below for more quantitative results, which give relations between the martingale cotype constant and the constants of the norm inequalities.

\section{Preliminaries}
We have gathered here some results from other papers which we will use. The following three lemmata on the spatial and temporal derivatives of the heat semigroup are from \cite{HN} (cf. Lemmata 25-27, pp. 22-24).

\begin{lemma} \label{div} Let $p \in [1, \infty],$ and let $Y$ be a Banach space. Then for any $\vec{f}= (f_1, \dots, f_n), $ $f_j \in L^p(\R^n;Y)$ we have
\begin{align*}
\| \sqrt{t} \, \text{\emph{div}} T_t \vec{f} \|_{L^p} \lesssim \sqrt{n} \fint_{S^{n-1}} \| \sigma \cdot \vec{f} \|_{L^p} \, d\sigma.
\end{align*}
\end{lemma}
\begin{lemma} \label{nabla}
Let $p \in [1, \infty],$ and let $Y$ be a Banach space. Then for any $f \in L^p(\R^n;Y),$ and any $z \in \R^n$ we have
\begin{align*}
\| \sqrt{t} \, (z \cdot \nabla) T_t f \|_{L^p} \lesssim |z| \| f \|_{L^p}.
\end{align*}
\end{lemma}
\begin{lemma} \label{dt}
Let $p \in [1, \infty],$ and let $Y$ be a Banach space. Then for any $f \in L^p(\R^n;Y),$ and any $z \in \R^n$ we have
\begin{align*}
\| t \,\dot{T}_t f \|_{L^p} \lesssim \sqrt{n} \, \| f \|_{L^p}.
\end{align*}
\end{lemma}

We will also need the  following vector-valued Littlewood-Paley-Stein inequalities from \cite{HN} (cf.  Theorems 17-18, p. 13).

\begin{theorem} \label{LPS}
Let $q \in [2,\infty),$ and $n \in \N.$ Suppose that $Y$ is a Banach space that admits  an equivalent norm with modulus of convexity of power type $q.$ Then for every $f\in L^q(\R^n;Y)$ we have
\begin{align*}
\left( \int_0^{\infty}  \|t \, \dot{T}_t f \|_{L^q}^q \frac{dt}{t}\right)^{1/q} \, \lesssim \mathfrak{m}(q,Y) \sqrt{n} \,  \| f \|_{L^q},
\end{align*}
and for every $\vec{f} \in l^n_q(L^q(\R^n;Y))$ we have
\begin{align*}
\left( \int_0^{\infty}  \| \sqrt{t} \,\text{\emph{div}} T_t \vec{f}\|_{L^q}^q  \frac{dt}{t}\right)^{1/q} \, \lesssim \mathfrak{m}(q,Y) \sqrt{n} \,  \| \vec{f} \|_{L^q}.
\end{align*}
\end{theorem}


\section{Besov Spaces and Local Approximation spaces} \label{localsection}

Let $Y$ be any Banach space. As mentioned above,  if $Y$ is $\R$ or $\C,$ then the Besov spaces $B^{sM}_{pq}(\R^n; Y),$ and the local approximation spaces $A^{sN}_{pq}(\R^n; Y)$ are equal, and their norms are equivalent, provided that $M$ and $N$ are large enough compared to $s.$ Our goal in this section is to show that this result holds also when the target space $Y$ is any Banach space. 

We will need  the following lemma, which essentially states that the norms  $[ \cdot ]_{B^{sM}_{pq}} $ and $[ \cdot ]_{B^{sK}_{pq}} $ are equivalent provided that $M$ and $K$ are large enough compared to $s.$

\begin{lemma} \label{ip} For all $M > s/2,$ and $f \in L^p(\R^n; Y)$ we have  $[ f ]_{B^{sM}_{pq}} \lesssim_{s,M} [ f ]_{B^{s,M+1}_{pq}}.$ Similarly,  for all $M \geq 1,$ and $f \in L^p(\R^n; Y)$ we have  $[ f ]_{B^{sM}_{pq}} \lesssim_{s,M} \sqrt{n} \, [ f ]_{B^{s,M-1}_{pq}}.$ 
\end{lemma}
\begin{proof}
Since $f \in L^p,$ we have
\begin{align*}
\left (\int_0^{\infty}  t^{(M-s/2)q}\left \|   T_t^{(M)} f  \right \|_{L^p}^q \frac{dt}{t} \right)^{1/q} & = \left (\int_0^{\infty}  t^{(M-s/2)q}\left \|  \int_t^{\infty} T_u^{(M+1)} f  du \right \|_{L^p}^q  \frac{dt}{t} \right)^{1/q}  \\
& \leq \int_1^{\infty} \left (\int_0^{\infty}  t^{(M +1-s/2)q}\left \|   T_{vt}^{(M+1)} f   \right \|_{L^p}^q  \frac{dt}{t} \right)^{1/q} dv \\
& =\int_1^{\infty} v^{s/2-M-1} dv \left (\int_0^{\infty}  t^{(M +1-s/2)q}\left \|   T_t^{(M+1)} f   \right \|_{L^p}^q  \frac{dt}{t} \right)^{1/q}
\end{align*}
where the integral converges since $M> s/2.$

For the second part we have by Lemma \ref{dt}, and the semigroup property of $T_t$
\begin{align*}
 [ f ]_{B^{sM}_{pq}} 
&= \left (\int_0^{\infty}  t^{(M-1-s/2)q}\left \|   t \dot{T}_{t/2}  ( T^{(M-1)}_{t/2} f) \right \|_{L^p}^q \frac{dt}{t} \right)^{1/q} \\
& \lesssim_{s, M}  \sqrt{n} \left (\int_0^{\infty}  t^{(M-1-s/2)q}  \left \|      T^{(M-1)}_{t} f \right \|_{L^p}^q\frac{dt}{t} \right)^{1/q}.
\end{align*}
 \end{proof}

By the above lemma, we may define  $B^{s}_{pq}(\R^n;Y) := B^{sM}_{pq}(\R^n;Y),$ where $M >s/2$ is any integer, where all the suitable norms $ \| \cdot \|_{B^{sM}_{pq}} $ are equivalent.
Our aim is to show Theorem \ref{AequalsB}. We do this in two parts; first we show the inequality  $[ \cdot ]_{A^{sN}_{pq}}  \lesssim [ \cdot ]_{B^{sM}_{pq}},$ where we have uniform dependence on $n$ for the implied constant.

In the proofs below we will need the difference operators
\begin{align*}
\Delta^1_h T_t &:= T_{t+h} -T_t, \\
 \Delta^m_h T_t &:= \sum_{j =0}^m (-1)^{m-j} \binom{m}{j} T_{t+hj} = \Delta^1_h \Delta^{m-1}_h T_t =  \Delta^{m-1}_h \Delta^1_h T_t.\\
\end{align*} 

\begin{prop} \label{BinA} Let Y be a Banach space, and let $s>0,$ $p \in [1,\infty],$ $q \in [1, \infty),$ $M > s/2,$ and let $N > s-1.$ Then for all $f \in L^p(\R^n; Y)$ we have $[ f ]_{A^{sN}_{pq}}  \lesssim_{s,N,M} [ f ]_{B^{sM}_{pq}},$ where the implied constant depends only on $s,$ $N,$ and $M.$
\end{prop}

\begin{proof} Set
\begin{align*}
 \tilde{\Delta}^{m}_h T_t &:=  \Delta^m_h T_t - (-1)^m T_t= \sum_{j =1}^m (-1)^{m-j} \binom{m}{j} T_{t+hj}.
\end{align*}

Let $K$ be the smallest integer $> s/2.$ Let $A = A(s):= 2K-1,$ if $\lfloor s \rfloor$ is odd, $A := 2K -2,$ if $\lfloor s \rfloor$ is even. Then $s < A+1 \leq N + 1.$ As a candidate for the polynomials $P$, $\deg P \leq N,$ in the definition of $\| f \|_{A^{sN}_{pq}},$ we choose $(-1)^{K}\text{Taylor}_x^{A} ( \tilde{\Delta}^{K}_{t^2} T_0 f),$ which is the order  $A$  Taylor expansion of the smooth function $(-1)^{K}\tilde{\Delta}^{K}_{t^2} T_0 f$ developed at $x.$ Then
\begin{align*}
[ f ]_{A^{sN}_{pq}}  \leq &  \left(  \int_0^{\infty} t^{-sq}  \left \| \left(  \fint_{B(x,t)} \| f(y)-(-1)^{K}\text{Taylor}_x^{A} ( \tilde{\Delta}^{K}_{t^2} T_0 f)(y) \| _Y^p dy \right)^{1/p} \right \|_{L^p}^q \frac{dt}{t} \right) ^{1/q} \\
 \leq &  \left(  \int_0^{\infty} t^{-sq}  \left \| \Delta^K_{t^2} T_0 f \right \|_{L^p}^q \frac{dt}{t} \right) ^{1/q} \\
 &+ \left(  \int_0^{\infty} t^{-sq}  \left \| \left(  \fint_{B(x,t)} \| \tilde{\Delta}^{K}_{t^2} T_0 f(y) -\text{Taylor}_x^{A} ( \tilde{\Delta}^{K}_{t^2} T_0 f)(y) \| _Y^p dy \right)^{1/p} \right \|_{L^p}^q \frac{dt}{t} \right) ^{1/q} \\
 =&:\, I_1 + I_2.
\end{align*}
The first integral is easy to estimate; using relations $\Delta^m_h = \Delta^1_h\Delta^{m-1}_h,$ and
\begin{align*}
\Delta^1_h T_t f = \int_t^{t+h} \dot{T}_u f du,
\end{align*} 
we obtain
\begin{align*}
\Delta^K_{t^2} T_0 f = t^{2K} \int_0^1 \cdots \int_0^1 T^{(K)}_{t^2(u_1 + \cdots+ u_K)} f du_1 \cdots du_{K}.
\end{align*}
Hence, we have
\begin{align*}
I_1 & \leq \int_0^1 \cdots \int_0^1  \left(  \int_0^{\infty} t^{(2K-s)q}  \left \|  T^{(K)}_{t^2(u_1 + \cdots +u_K)} f \right \|_{L^p}^q \frac{dt}{t} \right)^{1/q} du_1 \cdots du_{K} \\
& \leq  \int_0^1 \cdots \int_0^1 (u_1+ \dots +u_K)^{s/2 - K}  du_1 \cdots du_{K} \left(  \int_0^{\infty} t^{(K-s/2)q}  \left \|  T^{(K)}_{t} f \right \|_{L^p}^q \frac{dt}{2t} \right)^{1/q} \\
&  \leq  \int_{[0,1]^{K}} | u |^{s/2-K} du \left(  \int_0^{\infty} t^{(K-s/2)q}  \left \|  T^{(K)}_{t} f \right \|_{L^p}^q \frac{dt}{2t} \right)^{1/q} 
 \lesssim_{s, K}  [ f ]_{B^{sK}_{pq}},
\end{align*}
 since $|u| \leq u_1 + \cdots + u_K,$ and $K > s/2 > 0.$

For the second integral we use the integral formula for the error term in Taylor approximation to obtain
\begin{align*}
I_2 &=   \left(  \int_0^{\infty} t^{-sq}  \left \| \left(  \fint_{B(0,1)} \| \tilde{\Delta}^{K}_{t^2} T_0 f(x+tz) -\text{Taylor}_x^{A} ( \tilde{\Delta}^{K}_{t^2} T_0 f)(x+tz) \| _Y^p dz \right)^{1/p} \right \|_{L^p}^q \frac{dt}{t} \right) ^{1/q} \\
&= \left(  \int_0^{\infty} t^{-sq}  \left \| \left(  \fint_{B(0,1)} \left\| \int_0^1 \frac{(tz\cdot \nabla)^{A+1}}{A!} \tilde{\Delta}^{K}_{t^2} T_0 f(x+rtz)(1-r)^{A}  dr \right \| _Y^p dz \right)^{1/p} \right \|_{L^p}^q \frac{dt}{t} \right) ^{1/q} \\
& \leq  \int_0^1 \frac{(1-r)^{A}}{A!}  \left(  \int_0^{\infty} t^{-sq}  \left \| \left(  \fint_{B(0,1)} \left\|  (tz\cdot \nabla)^{A+1} \tilde{\Delta}^{K}_{t^2} T_0 f(x+rtz)   \right \| _Y^p dz \right)^{1/p} \right \|_{L^p}^q \frac{dt}{t} \right) ^{1/q} dr \\
& \lesssim_{s,K} \sup_{\vert z \vert \leq 1} \left(  \int_0^{\infty} t^{-sq/2}  \left \|   t^{(A+1)/2}(z\cdot \nabla)^{A+1} \tilde{\Delta}^{K}_{t} T_0 f    \right \|_{L^p}^q \frac{dt}{t} \right) ^{1/q} 
\end{align*}
by translation invariance in the last step. We have
\begin{align*}
\tilde{\Delta}^{K}_{t} T_0 f  & = \Delta^K_t T_0 f - (-1)^K T_0 f \\
&= \int_0^{t} \cdots \int_0^{t} T^{(K)}_{v_1 + \cdots + v_K} f dv_1 \cdots dv_{K}  - \int_0^{\infty} \cdots \int_0^{\infty} T^{(K)}_{v_1 + \cdots + v_K} f dv_1 \cdots dv_{K}  \\
&= - t^K \int_{\Omega} T^{(K)}_{t(u_1 + \cdots + u_K)} f du_1 \cdots du_{K}, 
\end{align*}
where $\Omega := [0, \infty)^K \setminus [0,1)^K.$ Thus,
\begin{align*}
I_2 \lesssim_{s,K} &  \sup_{\vert z \vert \leq 1} \left(  \int_0^{\infty} t^{(K-s/2)q}  \left \|   t^{(A+ 1)/2}(z\cdot \nabla)^{A+1} \int_{\Omega} T^{(K)}_{t(u_1 + \cdots + u_K)} f du_1 \cdots du_{K}  \right \|_{L^p}^q \frac{dt}{t} \right) ^{1/q} \\
 \lesssim_{s,K} &\int_{\Omega} (u_1 + \cdots + u_K)^{s/2 - K -(A+1)/2 } du_1 \cdots du_{K} \,  \times \\ 
 & \times  \sup_{\vert z \vert \leq 1} \left(  \int_0^{\infty} t^{(K-s/2)q}  \left \|   t^{(A+ 1)/2}(z\cdot \nabla)^{A+1}  T^{(K)}_{t} f  \right \|_{L^p}^q \frac{dt}{t} \right) ^{1/q} \\
 \lesssim_{s, K}& \sup_{\vert z \vert \leq 1} \left(  \int_0^{\infty} t^{(K-s/2)q}  \left \|   t^{(A+ 1)/2}(z\cdot \nabla)^{A+1}  \underbrace{ T_{t} \cdots T_{t}}_{A+1} T^{(K)}_{t}  f  \right \|_{L^p}^q \frac{dt}{t} \right) ^{1/q} \\
\lesssim_{s,K} &  \left(  \int_0^{\infty} t^{(K-s/2)q}  \left \|  T^{(K)}_{t} f \right \|_{L^p}^q \frac{dt}{t} \right)^{1/q}  = [ f ]_{B^{sK}_{pq}},
\end{align*}
where we have used Lemma \ref{nabla} $(A+1)$ times to obtain the last inequality, and the identity  $T^{(K)}_{(A+2)t} = T_{t} \cdots T_{t} T^{(K)}_{t},$ which holds by the semigroup property of $T_t$. The integral over $\Omega$ converges, since $(A+1)/2 > s/2.$ Hence, for the smallest integer $K>s/2$ we have $[ f ]_{A^{sN}_{pq}}  \lesssim_{s, K} [ f ]_{B^{sK}_{pq}}.$ For general $M > s/2,$  the claim $[ f ]_{A^{sN}_{pq}}  \lesssim_{s, N, M} [ f ]_{B^{sM}_{pq}}$ now follows from the first part of Lemma \ref{ip}.
\end{proof}

Proposition \ref{BinA} states that we have a continuous emmbedding $B^{s}_{pq} (\R^n; Y)\hookrightarrow A^{sN}_{pq} (\R^n; Y)$ for all Banach spaces $Y,$ and the constant in $[ \cdot ]_{A^{sN}_{pq}}  \lesssim [ \cdot ]_{B^{sM}_{pq}}$ is uniform in $n.$ The next proposition shows that the converse embedding also holds, but this time the constant grows polynomially in $n.$ These propositions combined prove Theorem \ref{AequalsB}.

\begin{prop} \label{AinB} Let $Y$ be a Banach space and let $s> 0,$ $p \in [1, \infty]$ $q \in [1, \infty),$ $M > s/2,$ and $N \geq 0.$ Then for all $f \in L^p(\R^n; Y)$
 we have $ [ f ]_{B^{sM}_{pq}} \lesssim_{s,N,M} n^{c(N,M,s)} [ f ]_{A^{sN}_{pq}},$ for $c(N,M,s) = (M+N+s+1)/2,$ where the implied constant depends only on $s,$ $N,$ and $M,$ 
\end{prop}

\begin{proof}
We have
\begin{align*}
 T^{(M)}_t f &= \sum_{j_1=0}^{\infty}T^{(M)}_{2^{j_1}t} f - T^{(M)}_{2^{j_1+1}t} f = -\sum_{j_1=0}^{\infty}\Delta^{1}_{2^{j_1}t} T^{(M)}_{2^{j_1}t} f\\
&= -\sum_{j_1=0}^{\infty} \sum_{j_2=0}^{\infty}\Delta^{1}_{2^{j_1+ j_2}t} T^{(M)}_{2^{j_1+j_2}t} f - \Delta^{1}_{2^{j_1+j_2+1}t} T^{(M)}_{2^{j_1+j_2+1}t} f  \\
&= \sum_{j_1=0}^{\infty} \sum_{j_2=0}^{\infty}\Delta^{2}_{2^{j_1+ j_2}t} T^{(M)}_{2^{j_1+j_2}t} f = \cdots \\
&= (-1)^N\sum_{j_1=0}^{\infty} \cdots \sum_{j_N=0}^{\infty}  \Delta^{N}_{2^{j_1+ \cdots + j_N}t} T^{(M)}_{2^{j_1+ \cdots + j_N}t} f .
\end{align*}
 
Thus,
\begin{align*}
[ f ]_{B^{sM}_{pq}} &= \left (\int_0^{\infty}  t^{(M-s/2)q}\left \|   T_t^{(M)} f  \right \|_{L^p}^q \frac{dt}{t} \right)^{1/q} \\
& \leq  \sum_{j_1=0}^{\infty} \cdots \sum_{j_N=0}^{\infty}  \left (\int_0^{\infty}  t^{(M-s/2)q}\left \|   \Delta^{N}_{2^{j_1+ \cdots + j_N}t} T^{(M)}_{2^{j_1+ \cdots + j_N}t} f \right \|_{L^p}^q \frac{dt}{t} \right)^{1/q} \\ 
&=  \sum_{j_1=0}^{\infty} \cdots \sum_{j_N=0}^{\infty} 2^{(j_1+ \cdots + j_N)(s/2-M)} \left (\int_0^{\infty}  t^{(M-s/2)q}\left \|   \Delta^{N}_{t} T^{(M)}_{t} f \right \|_{L^p}^q \frac{dt}{t} \right)^{1/q}  \\
& \lesssim_{s,M,N} \left (\int_0^{\infty}  t^{(M-s/2)q}\left \| \Delta^N_{t} T^{(M)}_{t} f  \right \|_{L^p}^q \frac{dt}{t} \right)^{1/q},
\end{align*}
  by making the change of variables  $t \mapsto t 2^{-(j_1 + \cdots + j_N)}.$ By linearity of $T_t,$ change of variables $t \mapsto 2t,$ and Lemma \ref{dt} applied $M$ times, we obtain
\begin{align}
[ f ]_{B^{sM}_{pq}} & \lesssim_{s,M,N}\left (\int_0^{\infty}  t^{(M-s/2)q}\left \| T^{(M)}_t \Delta^N_{2t} T_{t} f  \right \|_{L^p}^q \frac{dt}{t} \right)^{1/q} \label{timeout} \\
& \lesssim_{M} n^{M/2} \left (\int_0^{\infty}  t^{-s/2q}\left \| \Delta^N_{2t} T_{t} f  \right \|_{L^p}^q \frac{dt}{t} \right)^{1/q}.  \nonumber
\end{align}
To avoid having exponential dependence on $n,$ we  divide the differences into smaller differences; for any positive integer $l$ we have
\begin{align*}
\Delta^N_{2lt^2} T_{lt^2} & =  \Delta^{N-1}_{2lt^2} \sum_{k_1=l}^{3l-1} \Delta^{1}_{t^2} T_{k_1 t^2} = \sum_{k_1=l}^{3l-1} \Delta^{1}_{t^2} \Delta^{N-1}_{2lt^2} T_{k_1 t^2}  \\
&= \sum_{k_1=l}^{3l-1}  \Delta^{1}_{t^2} \Delta^{N-2}_{2lt^2} \sum_{k_2 = 0}^{2l-1} \Delta^{1}_{t^2} T_{(k_1+k_2)t^2} = \sum_{k_1=l}^{3l-1} \sum_{k_2 = 0}^{2l-1} \Delta^{2}_{t^2} \Delta^{N-2}_{2lt^2} T_{(k_1+k_2)t^2} = \cdots \\
& =  \sum_{k_1=l}^{3l-1}\sum_{k_2 = 0}^{2 l-1} \cdots \sum_{k_N = 0}^{2l-1} \Delta^{N}_{t^2} T_{(k_1+ \cdots + k_N)t^2}  =: \sum_K \Delta^{N}_{t^2} T_{\vert K\vert t^2},
\end{align*}
where $K = (k_1, \dots, k_N )$ and $\vert K \vert = k_1+ \cdots + k_N.$ Thus, by a change of variables $t \mapsto l t^2$ we have
\begin{align*}
[ f ]_{B^{sM}_{pq}} & \lesssim_{s,N,M} n^{M/2} l^{-s/2} \sum_K \left (\int_0^{\infty}  t^{-sq}\left \| \Delta^{N}_{t^2} T_{\vert K\vert t^2}  f \right \|_{L^p}^q \frac{dt}{t} \right)^{1/q}.
\end{align*}
We have for all $\alpha>0$ (abusing the notation for the heat kernel)
\begin{align*}
T_{\alpha t^2} f (x) &= \int_{\R^n} k_{\alpha t^2} (y) f(x-y)\, dy= - \int_{\R^n}  \int_{\vert y \vert}^{\infty} \partial_u k_{\alpha t^2} (u) \, du f(x-y)\, dy \\
&= -\int_0^{\infty} \fint_{B(0,u)} f(x-y) \, dy  \, \partial_u k_{\alpha t^2} (u) \vert B(u) \vert\, du \\
& = -\int_0^{\infty} \fint_{B(0,rt\sqrt{\alpha} )} f(x-y) \, dy  \, \partial_r k_{1} (r) \vert B(r) \vert\, dr,
\end{align*}
where $\vert B(r) \vert$ is the volume of the ball $B(0,r).$ Hence, $\Delta^N_{t^{2}} T_{\vert K\vert t^{2}} f (x) $ is equal to
\begin{align*}
  -\int_0^{\infty} \left( \sum_{j=0}^N (-1)^{N-j} \binom{N}{j} \fint_{B\left(0,rt\sqrt{\vert K\vert+j} \right)} f(x-y) \, dy  \right)\, \partial_r k_{1} (r) \vert B(r) \vert\, dr.
\end{align*}
We now claim that
\begin{align*}
\sum_{j=0}^N (-1)^{N-j} \binom{N}{j} \fint_{B\left(0,rt\sqrt{\vert K\vert+j} \right)} f(x-y) \, dy
\end{align*}
is invariant under replacing $f(x-y)$ by $f(x-y)-p(x-y),$ where $p$ is any polynomial such that $\deg p < 2N,$ and $p$ can depend on the parameters $x,$ $t,$ $r,$ and $K.$ Clearly it suffices to show that
\begin{align} \label{odd}
\sum_{j=0}^N (-1)^{N-j} \binom{N}{j} \fint_{B\left(0,rt\sqrt{\vert K\vert+j} \right)} y_1^{\beta_1} \cdots y_n^{\beta_n} \, dy = 0
\end{align}
for $\vert \beta \vert:=\beta_1 + \cdots + \beta_n < 2N.$ If $\vert \beta \vert $  is odd, then the above holds trivially, since the integral vanishes. If $|\beta|$ is even, then
\begin{align*}
\fint_{B(0,rt\sqrt{\alpha} )} y_1^{\beta_1} \cdots y_n^{\beta_n} \, dy  &= \frac{1}{\vert B(1)\vert} \int_{S^{n-1}} \sigma_1^{\beta_1} \cdots \sigma_n^{\beta_n} \, d\sigma \int_0^{rt \sqrt{\alpha}} \frac{u^{\vert \beta \vert+ n-1}}{(rt \sqrt{\alpha})^n} du \\
&= C_{\beta, n} r^{\vert \beta \vert} t^{\vert \beta \vert} \alpha^{\vert \beta \vert/2} =:Q_\beta(\alpha).
\end{align*}
Thus, if $\vert \beta\vert/2 < N,$ and  $|\beta|$ is even, then 
\begin{align*}
\sum_{j=0}^N (-1)^{N-j} \binom{N}{j} \fint_{B\left(0,rt\sqrt{\vert K\vert+j} \right)} y_1^{\beta_1} \cdots y_n^{\beta_n} \, dy = (\Delta^N_1 Q_\beta)(\vert K \vert) = 0.
\end{align*}
Hence, we may replace $f(x-y)$ by $f(x-y)-p^{r,K}_{x,t}(x-y),$ where $p^{r,K}_{x,t}$ are any polynomials of degree $< 2N.$ Thus, by estimating the averages by the largest one we obtain
\begin{align*}
[ f ]_{B^{sM}_{pq}}  \lesssim&_{s,N,M} \, n^{M/2} l^{-s/2} \sum_K  \sum_{j=0}^N \binom{N}{j} \int_0^{\infty} \vert \partial_r k_{1} (r) \vert \vert B(r) \vert\,  \times \\
& \times \left (\int_0^{\infty}  t^{-sq}\left \| \fint_{B\left(x,rt\sqrt{\vert K\vert+j} \right)} f(y) - p^{r,K}_{x,t}(y) \, dy \right \|_{L^p}^q \frac{dt}{t} \right)^{1/q} dr \\
 \leq & \,  n^{M/2} l^{-s/2}  \sum_K  \sum_{j=0}^N \binom{N}{j} \int_0^{\infty} \vert \partial_r k_{1} (r) \vert \vert B(r) \vert  \,  \left(\frac{\vert K\vert+N}{\vert K\vert +j}\right)^{n/2}\times \\
& \times \left (\int_0^{\infty}  t^{-sq}\left \| \fint_{B\left(x,rt\sqrt{\vert K\vert+N} \right)} f(y) - p^{r,K}_{x,t}(y) \, dy \right \|_{L^p}^q \frac{dt}{t} \right)^{1/q} dr \\
 \leq & \, C_{n,M,N,s}  \left (\int_0^{\infty}  t^{-sq}\left \| \left( \fint_{B\left(x,t\right)} \| f(y) - p_{x,t}(y) \|_Y^p \, dy \right)^{1/p} \right \|_{L^p}^q \frac{dt}{t} \right)^{1/q},
\end{align*}
if we choose $ p^{r,K}_{x,t} = p_{x, rt\sqrt{\vert K\vert+N} },$ and make the changes of variables $rt \sqrt{\vert K\vert+N} \mapsto t.$  The constant $C_{n,M,N,s} $ is equal to
\begin{align*}
n^{M/2} l^{-s/2}\int_0^{\infty} \vert \partial_r k_{1} (r) \vert \vert B(r) \vert r^s \, dr \sum_K  \sum_{j=0}^N \binom{N}{j}  \left(\frac{\vert K\vert +N}{\vert K\vert +j}\right)^{n/2} (\vert K\vert +N)^{s/2}.
\end{align*}
We have
\begin{align*}
\int_0^{\infty} \vert \partial_r k_{1} (r) \vert \vert B(r) \vert r^s \, dr &= \frac{\vert B(1)\vert }{2} \int_0^{\infty} r^{n+s+1}  \frac{e^{-r^2/4}}{(4\pi)^{n/2}} \, dr \\
&= \frac{\vert B(1) \vert}{(4\pi)^{n/2}} 2^{n+s+1} \int_0^{\infty} u^{(n+s)/2} e^{-u} \, du \\
&\asymp_s \frac{\Gamma(n/2+s/2+1)}{\Gamma(n/2+1)} \asymp_s n^{s/2}.
\end{align*}
Hence, for $l = nN$ we have
\begin{align*}
C_{n,M,N,s}  & \lesssim_{s,N} n^{M/2}  \sum_{k_1=nN}^{3nN-1}\sum_{k_2 = 0}^{2nN-1} \cdots \sum_{k_N = 0}^{2nN-1} \left(1+\frac{N}{\vert K\vert}\right)^{n/2} (\vert K\vert +N)^{s/2} \\
&\lesssim_{s,N} n^{M/2+s/2} \left(1+\frac{1}{n}\right)^{n/2} \sum_{k_1=nN}^{3nN-1}\sum_{k_2 = 0}^{2 nN-1} \cdots \sum_{k_N = 0}^{2nN-1} 1 \\
& \lesssim_{s,N} n^{M/2+N+s/2},
\end{align*}
since each of the sums has $2nN$ terms. Therefore,  for any polynomials $p_{x,t}$ such that $\deg p_{x,t} < 2N$ we have
\begin{align*}
 [ f ]_{B^{sM}_{pq}}  \lesssim_{s,N,M} n^{M/2+N+s/2} \left (\int_0^{\infty}  t^{-sq}\left \| \left( \fint_{B\left(x,t\right)} \| f(y) - p_{x,t}(y) \|_Y^p \, dy \right)^{1/p}  \right \|_{L^p}^q \frac{dt}{t} \right)^{1/q}.
\end{align*}
 Thus, we have  $[ f ]_{B^{sM}_{pq}} \lesssim_{s,N,M} n^{c(2N-1,M,s)}  [ f ]_{A^{s,2N-1}_{pq}}$ for $c(2N-1, M,s) = M/2+N+s/2.$ This already implies the theorem for
\begin{align*}
 c(N,M,s) = \begin{cases} & (M+N+s+1)/2, \quad \text{if} \,\,  N \,\, \text{is odd} \\ & (M+N+s+3)/2, \quad \text{if} \,\, N \,\, \text{is even.} \end{cases}
\end{align*} 
This asymmetry stems from the fact that the equality (\ref{odd}) holds trivially for odd $|\beta|.$ This is because the heat kernel is an even function.
To mend this, and  show that the estimate $ c(N,M,s) = (M+N+s+1)/2$ holds also when $N$ is even, we must retain one spatial derivative of $T_t$ all the way back in (\ref{timeout}). This will give us a kernel that is an odd function, thus favouring polynomials of even degree. Hence, by Lemma \ref{dt} applied $(M-1)$ times, and Lemma \ref{div} applied once, we obtain
\begin{align*}
[ f ]_{B^{sM}_{pq}} & \lesssim_{s,M,N}\left (\int_0^{\infty}  t^{(M-s/2)q}\left \| T^{(M-1)}_t \Delta^N_{3t} \dot{T}_{2t} f  \right \|_{L^p}^q \frac{dt}{t} \right)^{1/q} \\
& \lesssim_{s,M} n^{(M-1)/2} \left (\int_0^{\infty}  t^{(1-s/2)q}\left \| \Delta^N_{3t} (\text{div} T_{t} (\nabla T_t f))  \right \|_{L^p}^q \frac{dt}{t} \right)^{1/q} \\
& \lesssim  n^{M/2} \sup_{\sigma \in S^{n-1}}\left (\int_0^{\infty}  t^{-sq/2}\left \|  \sqrt{t} \, \Delta^N_{3t} ((\sigma \cdot \nabla) T_{t} f) \right \|_{L^p}^q \frac{dt}{t} \right)^{1/q},
\end{align*}
which yields for any positive integer $l$ by similar calculations as before
\begin{align*}
[ f ]_{B^{sM}_{pq}} \lesssim_{s,M}n^{M/2} l^{(1-s)/2}  \sup_{\sigma \in S^{n-1}} \sum_K \left (\int_0^{\infty}  t^{-sq}\left \| t \Delta^N_{t^2} ((\sigma \cdot \nabla) T_{\vert K \vert t^2} f) \right \|_{L^p}^q \frac{dt}{t} \right)^{1/q}.
\end{align*}
We now have for any $\alpha > 0$
\begin{align*}
(\sigma \cdot \nabla) T_{\alpha t^2} f (x) &= -\frac{1}{2\alpha t^2}\int_{\R^n} (\sigma \cdot y )k_{\alpha t^2} (y) f(x-y)\, dy \\
& = \int_0^{\infty} \left( \frac{1}{2\alpha t^2} \fint_{B(0,rt\sqrt{\alpha} )} (\sigma \cdot y )f(x-y) \, dy\right)   \partial_r k_{1} (r) \vert B(r) \vert\, dr.
\end{align*}
Thus, $t \Delta^N_{t^2} ((\sigma \cdot \nabla) T_{\vert K \vert t^2} f)$ is equal to
\begin{align*}
\int_0^{\infty} \left( \sum_{j=0}^N (-1)^{N-j} \binom{N}{j} \frac{1}{2(\vert K\vert +j) t} \fint_{B(0,rt\sqrt{\vert K\vert +j} )} (\sigma \cdot y )f(x-y) \, dy \right)\, \partial_r k_{1} (r) \vert B(r) \vert\, dr.
\end{align*}
Arguing similarly as above, we can show that this is invariant under replacing $f(x-y)$ by $f(x-y)-p^{r,K}_{x,t}(x-y),$ where $p^{r,K}_{x,t}$ are any polynomials of degree $\leq 2N.$ Note that in contrast to (\ref{odd}), which is trivial when $\vert \beta \vert=\beta_1 + \cdots + \beta_n $ is odd, we now have to show that
\begin{align} \label{even}
\sum_{j=0}^N (-1)^{N-j} \binom{N}{j} \frac{1}{(\vert K\vert +j) t} \fint_{B\left(0,rt\sqrt{\vert K\vert+j} \right)} (\sigma \cdot y ) y_1^{\beta_1} \cdots y_n^{\beta_n} \, dy = 0
\end{align}
for $\vert \beta \vert \leq 2N.$ This holds trivially when $\vert \beta \vert$ is even, which allows us to take $\vert \beta \vert \leq 2N,$ and improve the resulting estimate for polynomials of even degree. For odd $| \beta | < 2N,$ and any $j=1,\dots, n$  we have
 \begin{align*}
 \frac{1}{\alpha t}\fint_{B(0,rt\sqrt{\alpha} )} y_j y_1^{\beta_1} \cdots y_n^{\beta_n} \, dy  &= \frac{1}{\alpha t \vert B(1)\vert} \int_{S^{n-1}} \sigma_j \sigma_1^{\beta_1} \cdots \sigma_n^{\beta_n} \, d\sigma \int_0^{rt \sqrt{\alpha}} \frac{u^{\vert \beta \vert+ n}}{(rt \sqrt{\alpha})^n} du \\
&= C_{\beta, j, n} r^{\vert \beta \vert+1} t^{\vert \beta \vert} \alpha^{(\vert \beta \vert-1)/2} =:Q_\beta(\alpha),
 \end{align*}
which is of degree $< N.$ Thus, the differences guarantee that (\ref{even}) holds.  Therefore, for any polynomials $p^{r,K}_{x,t}$ of degree $\leq 2N$ we have
\begin{align*}
[ f ]_{B^{sM}_{pq}}  \lesssim&_{s,N,M} \, n^{M/2} l^{(1-s)/2} \sup_{\sigma \in S^{n-1}} \sum_K  \sum_{j=0}^N \binom{N}{j} \int_0^{\infty} \vert \partial_r k_{1} (r) \vert \vert B(r) \vert\,  \times \\
& \times \left (\int_0^{\infty}  t^{-sq}\left \| \fint_{B\left(x,rt\sqrt{\vert K\vert+j} \right)} \frac{\sigma \cdot (x-y)}{(\vert K\vert +j) t} \left( f(y) - p^{r,K}_{x,t}(y)\right) \, dy \right \|_{L^p}^q \frac{dt}{t} \right)^{1/q} dr  \\
& \leq C_{n,M,N,s} \left (\int_0^{\infty}  t^{-sq}\left \| \left( \fint_{B\left(x,t\right)} \| f(y) - p_{x,t}(y) \|_Y^p \, dy \right)^{1/p} \right \|_{L^p}^q \frac{dt}{t} \right)^{1/q},
\end{align*}
by estimating $|\sigma \cdot (x-y)| \leq |\sigma || x-y| \leq rt \sqrt{\vert K \vert + j},$ and choosing $ p^{r,K}_{x,t} = p_{x, rt\sqrt{\vert K\vert+N} }.$   The constant $C_{n,M,N,s}$ is equal to
\begin{align*}
n^{M/2} l^{(1-s)/2}\int_0^{\infty} \vert \partial_r k_{1} (r) \vert \vert B(r) \vert r^{s+1} \, dr \sum_K  \sum_{j=0}^N \binom{N}{j}  \left(\frac{\vert K\vert +N}{\vert K\vert +j}\right)^{n/2} \frac{(\vert K\vert +N)^{s/2}}{(|K| + j)^{1/2}}.
\end{align*}
By choosing $l = nN,$ and noting that $|K| \asymp_N n,$ and $\sum_K 1 \asymp_N n^N$, we obtain 
\begin{align*}
C_{n,M,N,s} \lesssim_{s,N,M} n^{M/2} n^{(1-s)/2} n^{(1+s)/2} n^N n^{(s-1)/2} = n^{M/2 +N + s/2 +1/2}.
\end{align*} 
Thus,  $[ f ]_{B^{sM}_{pq}} \lesssim_{s,N,M} n^{c(2N,M,s)}  [ f ]_{A^{s,2N}_{pq}}$ for $c(2N, M,s) = M/2+N+s/2 + 1/2.$  Thus, we may take $c(N, M,s) = (M+N+s + 1)/2.$
\end{proof}

\section{Besov Spaces and Sobolev Spaces of Integer Order} \label{sobolevsection}
In this section we study embeddings of Sobolev spaces into Besov spaces. Let $Y$ be a Banach space,  $k$ be a nonegative integer, and let   $q \in [2, \infty ).$
The main goal of this section is to show that $W^{k,q}(\R^n; Y ) \hookrightarrow B^k_{qq}(\R^n ; Y)$ continuously if and only if $Y$ has martingale cotype $q.$

Our work relies heavily on earlier work of Xu in \cite{Xu}, and Mart\'inez, Torrea and Xu in \cite{MTX}, as well as the results by Hyt\"onen and Naor in \cite{HN}. One of the main  results obtained in \cite{MTX} (cf. Theorem 5.2) can be restated as follows:
\begin{equation*}
  [f]_{\tilde B^0_{qq}(\R^n;Y)}\lesssim \| f \|_{W^{0,q}(\R^n;Y)}= \| f \|_{L^q(\R^n;Y)}
\end{equation*}
if and only if  $Y$ has martingale cotype $q$. The only  difference to our setting so far, signified by the tilde in $\tilde B^0_{qq}$, is that  instead of $T_t,$ they use the subordinated Poisson semigroup $P_t$  defined by
\begin{align*}
P_t f :=  \int_0^{\infty}  \frac{e^{-u}}{\sqrt{\pi u}} T_{t^2/4u} f du.
\end{align*}
 Using the Poisson semigroup, the homogeneous norm for the Besov space takes the form
\begin{align*}
 [ f ]_{\tilde{B}^{sK}_{pq}} := \left (\int_0^{\infty}  t^{(K-s)q}\left \|   P_t^{(K)} f  \right \|_{L^p}^q \frac{dt}{t} \right)^{1/q}.
\end{align*}
It is classical that for $Y=\R, \C$ we have  the equivalence of norms $[ \cdot ]_{\tilde{B}^{sK}_{pq}} \asymp [ \cdot ]_{B^{sM}_{pq}},$ provided that $M>s/2,$ and $K>s$ (cf. \cite{Tri}, pp. 151-155). For our immediate purposes, the following simple lemma suffices.

\begin{lemma} \label{ip2} Let $Y$ be a Banach space, $p,q \in [1, \infty),$ and let $M > (s-1)/2,$ $K>s.$ Then for all $f \in L^p$  we have $[ f ]_{\tilde{B}^{sK}_{pq}}   \lesssim_{s,M}  [ f ]_{\tilde{B}^{s,K+1}_{pq}},$ and $[ f ]_{\tilde{B}^{s,2M}_{pq}} \lesssim_{s,M} [ f ]_{B^{sM}_{pq} }.$ 
\end{lemma}

\begin{proof}
The proof of the first part is exactly the same as the proof of the first part of Lemma \ref{ip}. For the second part we note that $P^{(2M)}_t f = \Delta^MP_t f.$ Thus,
\begin{align*}
[ f ]_{\tilde{B}^{s2M}_{pq}} &=  \left (\int_0^{\infty}  t^{(2M-s)q}\left \|   P_t^{(2M)} f  \right \|_{L^p}^q \frac{dt}{t} \right)^{1/q} \\
& = \left (\int_0^{\infty}  t^{(2M-s)q}\left \|   \Delta^M P_t f  \right \|_{L^p}^q \frac{dt}{t} \right)^{1/q} \\
&= \left (\int_0^{\infty}  t^{(2M-s)q}\left \|   \int_0^{\infty}  \frac{e^{-u}}{\sqrt{\pi u}} \Delta^M  T_{t^2/4u} f du \right \|_{L^p}^q \frac{dt}{t} \right)^{1/q} \\
& \leq   \int_0^{\infty}  \frac{e^{-u}}{\sqrt{\pi u}}   \left (\int_0^{\infty}  t^{(2M-s)q}\left \|  \Delta^M  T_{t^2/4u} f \right \|_{L^p}^q \frac{dt}{t} \right)^{1/q}  du \\
& = \int_0^{\infty}  \frac{e^{-u}}{2\sqrt{\pi u}} (4u)^{M-s/2}  du   \left (\int_0^{\infty}  t^{(M-s/2)q}\left \|   T^{(M)}_{t} f \right \|_{L^p}^q \frac{dt}{t} \right)^{1/q} \\
& \lesssim_{s,M} [ f ]_{B^{sM}_{pq} }.
\end{align*}
\end{proof}
We also need the following lemma. For the proof we require the integral representation
\begin{align*}
P_t f = \frac{1}{\pi}\int_{\R}  \frac{t}{t^2+|y|^2} f(x-y) \,dy.
\end{align*}

\begin{lemma} \label{partial}
For all $k \geq 0$ and for all $f \in W^{k,q} (\R ; Y)$  we have
\begin{align*}
\left (\int_0^{\infty}  t^{q}\left \|   \dot{ P}_t \partial_x^k f  \right \|_{L^p}^q \frac{dt}{t} \right)^{1/q} \lesssim \left (\int_0^{\infty}  t^{q}\left \|   P_t^{(k+1)} f  \right \|_{L^p}^q \frac{dt}{t} \right)^{1/q}.
\end{align*}
\end{lemma}
\begin{proof}
If $k$ is even, then $P_t^{(k+1)} f = \dot{ P}_t \partial_x^k f,$ and we have an equality. If $k$ is odd, then by writing $P_t^{(k+1)} f = \ddot{ P}_t \partial_x^{k-1} f$ the lemma is reduced to the case $k=1.$ By the first part of Lemma \ref{ip2} we have
\begin{align*}
\left (\int_0^{\infty}  t^{q}\left \|   \dot{ P}_t \partial_x f  \right \|_{L^p}^q \frac{dt}{t} \right)^{1/q} \lesssim \left (\int_0^{\infty}  t^{2q}\left \|  \ddot{ P}_{2t} \partial_x f  \right \|_{L^p}^q \frac{dt}{t} \right)^{1/q}.
\end{align*}
By the semigroup property of $P_t$ we have $\ddot{ P}_{2t} \partial_x f = \partial_x P_t \ddot{ P}_{t} f.$ Thus, the lemma follows once we show that $\| t \partial_x P_t f \|_{L^p} \lesssim \| f \|_{L^p}.$ We have
\begin{align*}
\| t \partial_x P_t f \|_{L^p} &= \left \|  \frac{1}{\pi} \int_{\R} \partial_x \frac{ t^2}{t^2+(x-y)^2} f(y) \,dy \right \|_{L^p} 
= \left \| \frac{1}{\pi}  \int_{\R}  \frac{2 yt^2}{(t^2+y^2)^2} f(x-y) \,dy \right \|_{L^p} \\
& \lesssim \int_0^{\infty} \frac{2yt^2}{(t^2+y^2)^2} \,dy \, \| f \|_{L^p} =  \| f \|_{L^p},
\end{align*}
which completes the proof.
\end{proof}

We are now ready for the main theorem of this section, which states that a Banach space $Y$ has martigale cotype $q$ if and only if
$[f]_{B^{k}_{qq}(\R^n; Y}\lesssim\| f \|_{W^{k,q}(\R^n; Y)}$ for all $f\in W^{k,q}(\R^n;Y)$.
The proof of the direction  $(\Rightarrow)$ follows easily  from Littlewood-Paley-Stein inequalities proven in \cite{HN}. The proof of the direction $(\Leftarrow)$ follows by a reduction to the case $k=0$, which has already been established in \cite{MTX} and \cite{Xu}.


For the theorem, we define $\mathfrak{b}(k,q,n,Y)$ to be the best constant such that for all $f \in W^{k,q}(\R^n;Y)$
\begin{align}\label{eq:BvsW}
[ f ]_{B^{kM}_{qq}(\R^n;Y)}  \leq  \mathfrak{b}(k,q,n,Y) \| f \|_{W^{k,q}(\R^n;Y)},
\end{align}
where $M$ is the smallest integer $> k/2.$

\begin{theorem}  \label{mart}
Let $Y$ be a Banach space, and let $k \geq 0,$ $n \geq 1,$ $q \in [2, \infty).$ Then $Y$ has martigale cotype $q$ if and only if \eqref{eq:BvsW} holds.
Furthermore, 
\begin{align*}
  \mathfrak{b}(k,q,n,Y)  \lesssim  \mathfrak{m}(q,Y) \sqrt{n}, \quad \text{and} \quad \mathfrak{m}(q,Y) \lesssim \mathfrak{b}(k,q,1,Y) +C \lesssim \mathfrak{b}(k,q,n,Y) +C,
\end{align*}
 where the implied constants, and the constant $C$ do not depend on any of the parameters $k,q,n,Y$.
\end{theorem}

\begin{proof}
(i) $Y$ has  martingale cotype $q$ $\Rightarrow$ $\mathfrak{b}(k,q,n,Y)  \lesssim  \mathfrak{m}(q,Y) \sqrt{n}:$

 If $k$ is even, then  the smallest $M> k/2$  is $M = k/2+1,$ and we have by Theorem \ref{LPS}
 \begin{align*}
 \| f \|_{B^{kM}_{qq}} & = \left (\int_0^{\infty}  t^{q}\left \|   T_t^{(k/2+1)} f  \right \|_{L^q}^q \frac{dt}{t} \right)^{1/q} \\
 & =\left (\int_0^{\infty}  \left \|   t \dot{T}_t \Delta^{k/2} f  \right \|_{L^q}^q \frac{dt}{t} \right)^{1/q} \\
 & \lesssim \mathfrak{m}(q,Y)  \sqrt{n} \, \| \Delta^{k/2} f \|_{L^q} \leq \mathfrak{m}(q,Y)\sqrt{n} \,\| f \|_{W^{k,q}}.
 \end{align*}
 
 If $k$ is odd, then  the smallest $M> k/2$  is  $M=(k+1)/2,$ and we have by Theorem \ref{LPS}
 \begin{align*}
  \| f \|_{B^{kM}_{qq}} & = \left (\int_0^{\infty}  t^{q/2}\left \|   T_t^{(k+1)/2} f  \right \|_{L^q}^q \frac{dt}{t} \right)^{1/q} \\
 & =\left (\int_0^{\infty}  \left \|   t^{1/2} \dot{T}_t \Delta^{(k-1)/2} f  \right \|_{L^q}^q \frac{dt}{t} \right)^{1/q} \\
 & = \left (\int_0^{\infty}  \left \|   t^{1/2} \text{div} T_t (\nabla ( \Delta^{(k-1)/2} f )) \right \|_{L^q}^q \frac{dt}{t} \right)^{1/q} \\
 & \lesssim \mathfrak{m}(q,Y) \sqrt{n} \left(  \fint_{S^{n-1}} \| (\sigma \cdot \nabla) ( \Delta^{(k-1)/2} f ) \|_{L^q}^q d\sigma \right)^{1/q}  \\
 &\leq \mathfrak{m}(q,Y) \sqrt{n} \,\| f \|_{W^{k,q}}.
 \end{align*}
 
(ii) Reduction of dimension: 
$\mathfrak{b}(k,q,1,Y) \lesssim \mathfrak{b}(k,q,n,Y)$: 

We first require a smooth approximation of the function $N^{(1-n)/q} \chi_{[0,N]^{n-1}}$ on $\R^{n-1};$ let $\theta:\R \to \R$ be a smooth function such that $\theta(x) = 1$ for $ x \in[0,1],$ and $\text{supp}( \theta) \subset [-1,2].$ For $N \geq 1,$ let 
\begin{align*}
\theta_N(x) := \begin{cases} N^{-1/q}, &\text{ if } x \in [0,N] \\
N^{-1/q} \theta (x-N+1), &\text{ if } x \in [N,N+1] \\
N^{-1/q} \theta(x), &\text{ if } x \in [-1,0], \\
0, &\text{ if } x \notin [-1,N+1].
\end{cases}
\end{align*}
Denote $\tilde{x} := (x_2,x_3, \dots x_n)\in \R^{n-1},$ and define $\phi_N(\tilde{x}) := \theta_N(x_2)\theta_N(x_3) \cdots \theta_N(x_n).$ It is then clear that for all $t \in [0,1],$ $\alpha=1,2, \dots, 2M,$ and $j=1,2, \dots, M$ 
\begin{align*}
& \| \phi_N\|_{L^q} \to 1, \quad \quad \|T_t \phi_N\|_{L^q} \to 1, \\
& \| \partial_x^{\alpha} \phi_N\|_{L^q} \to 0, \quad  \| T^{(j)}_{t} \phi_N\|_{L^q} \leq \| \Delta^j \phi_N\|_{L^q} \to 0,
\end{align*}
uniformly as $N \to \infty.$ Note that above $T_t$ is the heat semigroup on $\R^{n-1}.$ Below the heat semigroup is always of the same dimension as the domain of the function it acts on.

Let $f \in W^{k,q}(\R;Y).$ Write $(f\phi_N) (x) := f(x_1) \phi_N(\tilde{x})$ for $x \in \R^{n}.$  We have $T_t (f\phi_N) =T_t (f)T_t (\phi_N).$ Thus, for large enough $N,$ and for all $t \in [0,1]$ we have
\begin{align}
(1+o(1)) \|T_t^{(M)}f \|_{L^q} &= \|T^{(M)}_t f\|_{L^q}\|T_t \phi_N\|_{L^q} \nonumber \\
& \leq \| T^{(M)}_t (f \phi_N)\|_{L^q} + \sum_{j<M} \binom{M}{j}\|T^{(j)}_t f\|_{L^q}\|T^{(M-j)}_t \phi_N\|_{L^q} \nonumber \\
& \leq \| T^{(M)}_t (f \phi_N)\|_{L^q} + \sum_{j<M} \binom{M}{j}\|T^{(j)}_t f\|_{L^q}\| \Delta^{(M-j)} \phi_N\|_{L^q}. \label{helpful}
\end{align}

Let $j < M.$ If $j \leq k/2,$ then $\|T^{(j)}_t f\|_{L^q}\| \leq \|\Delta^{j} f\|_{L^q} \leq \|f\|_{W^{k,q}},$ and we find that
\begin{align} \label{jestimate}
\left( \int_0^1 t^{(M-k/2)q} \|T^{(j)}_tf\|^q_{L^q} \frac{dt}{t} \right)^{1/q} \lesssim \|f\|_{W^{k,q}}.
\end{align}
If $j>k/2,$ then for even $k$ we use Lemma \ref{dt} $(j-k/2)$ times to obtain
\begin{align*}
\left( \int_0^1 t^{(M-k/2)q} \|T^{(j)}_tf\|^q_{L^q} \frac{dt}{t} \right)^{1/q} \lesssim \left( \int_0^1 t^{(M-j)q} \|T^{ (k/2) }_t f\|^q_{L^q} \frac{dt}{t} \right)^{1/q} \lesssim \|f\|_{W^{k,q}}.
\end{align*}
For $j>k/2,$ and odd $k$ we write $T^{(k)}_t f = \text{div}T^{(k-1)} (\partial_x f),$ and use Lemmata \ref{div} and \ref{dt} to obtain
\begin{align*}
\left( \int_0^1 t^{(M-k/2)q} \|T^{(j)}_tf\|^q_{L^q} \frac{dt}{t} \right)^{1/q} \lesssim \left( \int_0^1 t^{(M-j)q} \|T^{ ((k-1)/2) }_t \partial_x f\|^q_{L^q} \frac{dt}{t} \right)^{1/q} \lesssim \|f\|_{W^{k,q}}.
\end{align*}  
Thus, the estimate (\ref{jestimate}) holds for all $j<M.$  Hence, by integrating both sides of (\ref{helpful}), and using (\ref{jestimate}) we get
\begin{align*}
 [ f ]_{B^{kM}_{qq}} &\lesssim \|f\|_{L^q} + \left( \int_0^1 t^{(M-k/2)q} \|T^{(M)}_tf\|^q_{L^q} \frac{dt}{t} \right)^{1/q} \\
& \lesssim \|f\|_{L^q}  + \|f\phi_N\|_{B^{kM}_{qq}} + \sum_{j<M} \binom{M}{j}\| f\|_{W^{k,q}}\| \Delta^{(M-j)} \phi_N\|_{L^q} \\
& \lesssim \|f\|_{L^q}  + \|f\phi_N\|_{W^{k,q}} + \| f\|_{W^{k,q}}\, \lesssim \, \| f\|_{W^{k,q}}.
\end{align*}
By letting $N$ be large, we see that the implied constant of $\|f\|_{B^{kM}_{qq}} \lesssim \| f\|_{W^{k,q}}$ is $\mathfrak{b}(k,q,n,Y)$ up to an absolute constant.
  
(iii) Reduction of smoothness: 
$[f]_{B^{kM}_{qq}(\R;Y)}\lesssim \| f \|_{W^{k,q}(\R^n;Y)}$ $\Rightarrow$  $[f]_{\tilde B^{01}_{qq}(\R;Y)}\lesssim \| f \|_{L^q(\R^n;Y)}$:

 Using Lemmata \ref{partial} and \ref{ip2} we now obtain
 \begin{align*}
 \left (\int_0^{\infty}  t^{q}\left \|   \dot{ P}_t \partial_x^k f  \right \|_{L^q}^q \frac{dt}{t} \right)^{1/q} & \lesssim \| f \|_{\tilde{B}^{k,k+1}_{qq}}  \\ 
 & \lesssim \begin{cases}   & \| f \|_{\tilde{B}^{k,k+1}_{qq}}, \text{ if $k$ is odd,} \\  
 & \| f \|_{\tilde{B}^{k,k+2}_{qq}}, \text{ if $k$ is even,} 
  \end{cases} \\
  & \lesssim \begin{cases}   & \| f \|_{B^{k,(k+1)/2}_{qq}}, \text{ if $k$ is odd,} \\  
 & \| f \|_{B^{k,k/2+1}_{qq}}, \text{ if $k$ is even,} 
  \end{cases}  \\
  & \lesssim \| f \|_{W^{k,q}} = \sum_{0 \leq   \alpha  \leq k} \| \partial_x^{\alpha} f\|_{L^q}.
 \end{align*}
 If we substitute $f(\lambda\cdot)$ in place of $f$ and use the easy scaling properties $\partial_x^\alpha[f(\lambda\cdot)]=\lambda^\alpha(\partial_x^\alpha f)(\lambda\cdot)$, $ t\dot P_t[f(\lambda\cdot)]=\lambda t(\dot P_{\lambda t}f)(\lambda\cdot)$ and $\| f(\lambda\cdot)\|_{L^q}=\lambda^{-1/q}\| f\|_{L^q}$, we arrive at
 \begin{equation*}
  \lambda^{k-1/q} \left (\int_0^{\infty}  t^{q} \| \dot{ P}_t \partial_x^k f (x) \|_{L^q} ^q \frac{dt}{t} \right)^{1/q} \lesssim \sum_{0 \leq \alpha  \leq k}
    \lambda^{\alpha-1/q} \| \partial_x^{\alpha} f\|_{L^q}. 
\end{equation*}
Multiplying both sides by $\lambda^{1/q-k}$ and letting $\lambda\to\infty$, we find that all terms except $\|\partial_x^k f\|_{L^q}$ disappear on the right, so that in fact
\begin{equation}\label{eq:MTXembedding}
   \left (\int_0^{\infty}  t^{q} \| \dot{ P}_t g \|_{L^q} ^q \frac{dt}{t} \right)^{1/q} \lesssim  \| g\|_{L^q}
\end{equation}
for all functions of the form $g=\partial_x^k f$ for some $f\in W^{k,q}(\R;Y)$. We note that such functions are dense in $L^q(\R;Y)$; for instance, this follows from the density of smooth functions whose Fourier transform is compactly supported away from the origin (see \cite{HNVW}, Proposition 2.4.23(ii)) and the easy observation that each such function can be expressed in the required form. From the fact that \eqref{eq:MTXembedding} holds in a dense class of functions $g\in L^q(\R;Y)$ and the $L^q(\R;Y)$-boundedness of $\dot P_t$ for each fixed $t$, it is easy to deduce that \eqref{eq:MTXembedding} extends to all functions $g\in L^q(\R;Y)$. But this is precisely the claimed embedding $L^{q}(\R; Y) \hookrightarrow \tilde B^{01}_{qq}(\R; Y)$.

(iv) $[f]_{\tilde B^{01}_{qq}(\R;Y)}\lesssim \| f \|_{L^q(\R^n;Y)}$ $\Rightarrow$ $Y$ has  martingale cotype $q:$  This is a combination of earlier results by Mart\'inez, Torrea and Xu \cite{MTX} and Xu \cite{Xu}: Theorem 5.2 of \cite{MTX} guarantees that $[f]_{\tilde B^{01}_{qq}(\R;Y)}\lesssim \| f \|_{L^q(\R^n;Y)}$ if and only if $Y$ has so-called Lusin cotype $q$, and Theorem 3.1 of \cite{Xu} gives the equivalence of Lusin cotype $q$ and martingale cotype $q$, with quantitative bounds for the relevant constants in both cases. We refer to the cited papers for the definition of Lusin cotype, as we only need this intermediate notion as a black box in the mentioned implications.
\end{proof}

By Theorem \ref{AequalsB}, we now obtain the following corollary. For the corollary, we define $\mathfrak{a}(k,N,q,n,Y)$ to be the best constant such that
\begin{align*}
\| f \|_{A^{kN}_{qq}(\R^n;Y)}  \leq  \mathfrak{a}(k,N,q,n,Y) \| f \|_{W^{k,q}(\R^n;Y)}
\end{align*}
holds. 
\begin{cor} \label{mart2}
Let $Y$ be a Banach space, and let $k \geq 1,$  $q \in [2, \infty).$ Then $Y$ has martigale cotype $q$ if and only if $W^{k,q}(\R^n; Y) \hookrightarrow A^{k}_{qq}(\R^n; Y)$ continuously. Furthermore, if either one holds, then 
\begin{align*}
&\mathfrak{a}(k,N,q,n,Y) \lesssim_{k,N}  \mathfrak{m}(q,Y) \sqrt{n}, \quad {and} \\ & \mathfrak{m}(q,Y)  \lesssim_{k,N} \mathfrak{a}(k,N,q,1,Y) + C \lesssim_{k,N} \mathfrak{a}(k,N,q,n,Y) + C
\end{align*}
 for all $N \geq k.$ Here $C$ is the same absolute constant as in Theorem \ref{mart}.
\end{cor}

\begin{remark}
Theorem \ref{mart} may also be compared with \cite[Proposition 3.1]{Veraar}, which investigates conditions for the embedding of the Bessel potential space
\begin{equation*}
  H^{s,q}(\R^n;Y):=\{f\in L^q(\R^n;Y):\|f\|_{H^{s,q}}:=\|(1-\Delta)^{s/2}f\|_{L^q}<\infty\}
\end{equation*}
into $B^s_{qq}(\R^n;Y)$. It is shown in \cite{Veraar} that this happens when $Y$ is a UMD space with Rademacher cotype $q\in[2,\infty)$. See \cite{Veraar} or \cite{HNVW} for the definitions of these notions.

It is a well-known consequence of Fourier multiplier theorems valid in UMD spaces that $H^{k,q}(\R^n;Y)=W^{k,q}(\R^n;Y)$ when $Y$ is a UMD space (see \cite[Theorem 5.6.11]{HNVW}). On the other hand, Rademacher cotype $q$ and martingale cotype $q$ are equivalent properties of $Y$, when $Y$ is a UMD space (see \cite[Proposition 4.3.13]{HNVW}). Thus the mentioned result of \cite{Veraar}, in the case of integer smoothness $s=k$, can be seen as a corollary of Theorem \ref{mart}. In particular, the assumption that a space is UMD with Rademacher cotype $q$ is strictly stronger than martingale cotype $q$.

It is also asked in \cite[Remark 3.2]{Veraar} whether the assumption of Rademacher cotype is necessary. By the equivalences just explained, Theorem \ref{mart} answers this affirmatively, at least in the case of integer smoothness $s=k$.

Finally, \cite{Veraar} also proves a reverse embedding $B^s_{pp}(\R^n;Y)\hookrightarrow H^{s,p}(\R^n;Y)$ when $Y$ is a UMD space with Rademacher type $p\in(1,2]$, and asks about the necessity of the type $p$ condition. We suspect that there is a corresponding reverse analogue of Theorem \ref{mart}, namely the equivalence of martingale type $p$ with the embedding $B^k_{pp}(\R^n;Y)\hookrightarrow W^{k,p}(\R^n;Y)$, which would in particular answer the last mentioned question of \cite{Veraar} in the affirmative. However, we have not pursued this line of inquiry here.
\end{remark}

\subsubsection*{Acknowledgement}
We would like to thank an anonymous referee for constructive comments that improved this paper.

\bibliography{besovembedding}
\bibliographystyle{abbrv}

%
%
%
%
%
%
%
%
%
%

\end{document}